\documentclass[12pt]{amsart}
\usepackage{amsfonts}
\usepackage[utf8]{inputenc}
\usepackage{amssymb, amstext, amscd, amsmath}
\usepackage{amsthm}
\usepackage{times}
\usepackage{txfonts}
\usepackage{indentfirst}
\usepackage{subfigure}
\usepackage{graphicx}
\usepackage{mathrsfs}
\usepackage{geometry}

\usepackage{url}
\usepackage[colorlinks = true,
			linkcolor=blue,
			urlcolor=blue,
			citecolor=blue,
			anchorcolor=blue,
			backref=page]{hyperref}
\usepackage{color,xcolor}

\usepackage{color,xcolor}

\theoremstyle{theorem}
\newtheorem{thm}{Theorem}[section]
\newtheorem{cor}[thm]{Corollary}
\newtheorem{prop}[thm]{Proposition}
\newtheorem{lem}[thm]{Lemma}

\theoremstyle{definition}
\newtheorem{dfn}[thm]{Definition}
\newtheorem*{prb}{Problem}

\numberwithin{equation}{section}

\begin{document}

\title[Amalgamations along surfaces with boundary in a handlebody]{Amalgamations along surfaces with boundary \\ in a handlebody}

\author{Siqi Ding}
\address{School of Mathematical Sciences, Dalian University of Technology, Dalian 116024, CHINA}
\email{sqding@yeah.net}

\author{Fengchun Lei}
\address{School of Mathematical Sciences, Dalian University of Technology, Dalian 116024, CHINA \&  BIMSA}
\email{fclei@dlut.edu.cn}

\author{Wei lin}
\address{Fujian Normal University}
\email{linwei201208@fjnu.edu.cn}

\author{Andrei Vesnin}
\address{Sobolev Institute of Mathematics, Novosibirsk, Russia}
\email{vesnin@math.nsc.ru}



\begin{abstract}
Let $M$ be a connected orientable 3-manifold, and $F$ a compact connected orientable surface properly embedded in $M$. If $F$ cuts $M$ into two connected 3-mani\-folds $X$ and $Y$, that is, $M=X\cup_F Y$, we say that $M$ is an \textit{amalgamation} of $X$ and $Y$ along $F$;  and if $F$ cuts $M$ into a connected 3-manifold $X$, we say that $M$ is a \textit{self-amalgamation} of $X$ along $F$. A characterization of an amalgamation of two handlebodies along a surface, incompressible in both, to be a handlebody was obtained by Lei, Liu, Li, and Vesnin. The case of amalgamation of two handelbodies along a compressional surface was studdied by Xu, Fang, and Lei.  In the present paper, a characterization of an amalgamation and self-amalgamation of a handlebody to be a handlebody is given.
\end{abstract}

\thanks{S.~D., F.~L. and W.~L. supported in part by a grant (No.12071051) of NSFC; A.~V. supported by the State Task to the Sobolev Institute of Mathematics (No. FWNF-2022-0004)}
\subjclass[2010]{57N10}
\keywords{Heegaard splitting; $H'$-splitting; handlebody; incompressible surface; primitivity; quasi-primitivity}

\maketitle



\section{Introduction} \label{sec1}

Let $M$ be a compact connected orientable 3-manifold, and $F$ a compact connected orientable surface properly embedded in $M$. Denote by $\mathcal M$ the manifold obtained by cutting $M$ along $F$. We say that $M$ is an \textit{amalgamation} of $\mathcal M$ along $F$, and $F$ is an \textit{amalgamatimg surface}. If $F$ is a separating surface in $M$, then $\mathcal M$ contains two components $M_1$ and $M_2$ and we write $M=M_1 \cup_F M_2$. In particular, when both $M_1=H_1$ and $M_2=H_2$ are handlebodies, we say that $M = H_1\cup_F H_2$ is a \textit{$H'$-splitting} of $M$; moreover, if $H_1$ and $H_2$ are of the same genus and $F = \partial H_1 = \partial H_2$ then $M = H_1 \cup_F H_2 $ is a Heegaard splitting of a closed 3-manifolds $M$ and $F$ is a Heegaard surface.  If $F$ is a non-separating surface in $M$, then $\mathcal M$ is connected, we  say that $M$ is a \textit{self-amalgamation} of $\mathcal M$ along the two cutting sections of $F$ lying in $\partial \mathcal M$ and $F$ is a \textit{self-amalgamating surface}.

In recent years, many studies have been done on the topics on amalgamations or self-amalgamations of 3-manifolds or Heegaard splittings, see for examples, \cite{D-Q, E-L, L-L, LT, L-L-L, L-L-L2, M-L-L, ZDGQ}.
It is widely known that any closed connected orientable 3-manifold admits a Heegaard splitting. It was shown in~\cite{Gao} that any compact connected 3-manifolds admits an $H'$-splitting.
In the present paper we address the following problem:
\begin{prb}
When do amalgamations of handlebodies become a handlebody?
\end{prb}

A necessary and sufficient condition for a  $H'$-splitting $H_1\cup_F H_2$ to be a handlebody has been given in~\cite[Theorem~3.2]{lei2020necessary}, where $F$ is incompressible in both $H_1$ and $H_2$. This result was extended in~\cite{XFL} for the case that $F$ could be compressible in $H_1$ (or $H_2$) and incompressible in $H_2$ (or $H_1$), or compressible in both $H_1$ and $H_2$.

In the present paper we provide a characterization of a self-amalgamation of a handlebody to be a handlebody. The main result if the following theorem.

\begin{thm} \label{thm1.1}
Let
\begin{itemize}
\item[(1)] $\mathcal{H} = \{H_1, H_2\}$ be a collection of handlebodies and $M = \mathcal H_F = H_1 \cup_F H_2$ an amalgamation of $\mathcal H$ along $F$, or
\item[(2) ]$\mathcal{H} = \{H\}$ be a handlebody $H$ and $M=\mathcal{H}_F$ a self-amalgamation of $\mathcal{H}$ along $F$.
\end{itemize}
Assume that $F$ is incompressible in $M$ and $F$ is not a disc.
Then $M$ is a handlebody if and only if there exists a collection $\mathcal{J}=\{J_1,\ldots,J_p\}$ of pairwise disjoint simple closed curves on $F$ and a collection ${\mathcal D}=\{D_1,\ldots,D_p\}$ of pairwise disjoint disks properly embedded in $\mathcal{H}$ such that $|J_i \cap \partial D_j|=\delta_{ij}$ for $1\leq i, j \leq p$, where $p=\operatorname{rank}(\pi_1(F))$.
\end{thm}

In particular, collection $\mathcal{J}=\{J_1,\ldots,J_p\}$ is primitive in~$\mathcal{H}$ , see definition~\ref{d2.8}.

The paper is organized as follows. In Section~\ref{sec2} we recall basic definitions and facts from topology of 3-manifolds.  In Sections~\ref{sec3} and~\ref{sec4} we develop a theory of incompressible and compressible surfaces in a handlebody, see Theorem~\ref{tt3.3}, Theorem~\ref{tt3.3}$'$ and Corollary~\ref{cc3.7}. In Section~\ref{sec5} we discuss the case when the amalgamating surface is an annulus, see Theorem~\ref{tt3.9}. In section~\ref{sec6} we give a proof of Theorem~\ref{thm1.1}. Finally, in Section~\ref{sec7} we prove the existence of maximal system of compression discs, see Theorem~\ref{tt3.11}.

\section{Preliminaries} \label{sec2}

In this section, we recall necessary definitions and fundamental results about 3-manifolds that will be used below. The concepts and terminologies that are not defined in the paper are all standard in 3-manifold topology, see~\cite{hatcher2007notes, hempel20043, jaco1980lectures}, and combinatorial group theory, see~\cite{lyndon1977combinatorial}. All surfaces and 3-manifolds considered in the paper are assumed to be orientable. For a $(k-1)$-dimensional sub-manifold $Y$ properly embedded in a $k$-manifold $X$ with a regular neighbourhood $N(Y)$ in $X$ ($k=2$ or 3), we use $X\setminus Y$ to denote the $n$-manifold $\overline{X-N(Y)}$, and call it the manifold obtained by cutting $X$ open along $Y$.

Let $M$ be a compact connected orientable 3-manifold with boundary $\partial M$. If there exists a collection $\mathcal{D}=\{D_1,\ldots,D_n\}$ of $n\geq 0$ pairwise disjoint disks properly embedded in $M$, such the 3-manifold obtained by cutting $M$ open along $\mathcal{D}$ is a 3-ball, we call $M$ a \textit{handlebody}, and call $n$ the \textit{genus} of $M$, with notation $g(M)=n$.

Let $F$ be a compact surface, and $\alpha$ a simple arc or a simple closed curve properly embedded in $F$. We say that $\alpha$ is \textit{trivial} if $\alpha$ cuts out of a disk from $F$. Otherwise it is \textit{non-trivial}. If a simple closed curve $\alpha$ in $F$ is neither trivial nor boundary parallel in $F$, we say $F$ is \textit{essential}.

A $(k-1)$-dimensional sub-manifold $Y$ properly embedded in a connected $k$-manifold $X$ is \textit{separating} if $X\backslash Y$ is not connected; otherwise it is \textit{non-separating}.

Let $M$ be a 3-manifold. Suppose $F$ is a subsurface of $\partial M$ or a surface properly embedded in $M$. If either
\begin{itemize}
\item[(1)] $F$ is a 2-sphere bounding a 3-ball in $M$, or $F$ is a disk lying in $\partial M$, or $F$ is a disk properly embedded in $M$ which cuts out of a 3-ball from $M$, or
\item[(2)] there is a non-trivial simple closed curve $\alpha$ in $F$ that $\alpha$ bounds a disk $D$ in $M$ with $D\cap F=\partial D$,
\end{itemize}
then $F$ is said to be \textit{compressible} in $M$, and $D$ is called a \textit{compression disk}. We say that $F$ is \textit{incompressible} in $M$ if $F$ is not compressible in $M$. The surface $F$ in (1) is also called a \textit{trivial compressible surface}.

Let $M$ be a 3-manifold, and $F$ a properly embedded surface in $M$. We say that $F$ is \textit{boundary compressible} (or \textit{$\partial$-compressible}) in $M$, if
\begin{itemize}
\item there is a disk $\Delta$ in $M$ such that $\Delta\cap F=\alpha$ is an essential arc in $F$, and
\item $\Delta\cap\partial M=\beta$ is a simple arc in $\partial M$,
\item with $\alpha \cap \beta=\partial\alpha=\partial\beta$ and $\alpha \cup \beta=\partial\Delta$.
\end{itemize}
$F$ is said to be \textit{$\partial$-incompressible} if it is not $\partial$-compressible in~$M$.

\begin{dfn}
Let $M$ be a compact connected oriented 3-manifold with non-empty boundary $\partial M$, and $F$ a compact orientable surface in $M$ such that each component $F_i$ of $F$ is a surface properly embedded in $M$ with non-empty boundary, $1\leq i\leq p$. Assume that the 3-manifold $\mathcal M$, obtained by cutting $M$ open along $F$, has $q$ connected components $M_1,\ldots, M_q$. Write $\mathcal{M}=\{M_1,\ldots, M_q\}$, and denote the two cutting sections of each $F_i$ by $F_i'$ and $F_i''$, respectively, $1\leq i\leq p$. Then $M$ can be recovered by gluing $M_1,\ldots, M_q$ together through $\mathcal F = \{f_j \mid 1\leq j\leq p\}$, where each $f_j:F_j'\to F_j''$ is an orientation-reversing homeomorphism, $1\leq j\leq p$.
Manifold $M$ is called an \textit{amalgamation} of $\mathcal{M}$ via $\mathcal F$, and is denoted by $\mathcal{M}(\mathcal F)$. We call $F$ an \textit{amalgamated} surface of $\mathcal{M}$ (or in $M$). Manifold $M$ can also be denoted by $\mathcal{M}_F$.

In particular, we get:
\begin{itemize}
\item If $F$ is connected and separating in $M$, then $F$ cuts $M$ into $M_1$ and $M_2$, we say that $M$ is an \textit{amalgamation} of $M_1$ and $M_2$ along $F'$ and $F''$. If both manifolds are handlebodies, $M_1 = H_1$ and $M_2= H_2$, then  $F$ is a \textit{$H'$-surface} in $M$ and $H_1\cup_F H_2$ \textit{is} a \textit{$H'$-splitting} for $M$. Moreover, if $F=\partial H_1=\partial H_2$, then $F$ is a Heegaard surface in $M$ and $H_1\cup_F H_2$ is a \textit{Heegaard splitting} for $M$.
\item If $F$ is connected and non-separating in $M$, then $F$ cuts $M$ into a connected 3-manifold $\mathcal M = \{ M_1 \}$, we say that $M$ is an \textit{self-amalgamation} of $\mathcal M$ along $F'$ and $F''$.
\end{itemize}
\end{dfn}

\begin{dfn}
Let $M$ be a compact connected orientable 3-manifold, and $F$ a compact connected orientable (equivalently, 2-sided) surface properly embedded in $M$. Suppose that there exist compression disks $D_1$ and $D_2$ of $F$ which are lying in the two sides of $F$, respectively.
\begin{itemize}
\item If $\partial D_1=\partial D_2$, then we say that $F$ is \textit{reducible}; otherwise it is \textit{irreducible};
\item if $\partial D_1\cap\partial D_2=\emptyset$, then we say that $F$ is \textit{weakly reducible}; otherwise it is \textit{strongly irreducible};
\item if $\partial D_1$ intersects $\partial D_2$ in a single point, then we say that $F$ is \textit{stabilized}; otherwise it is \textit{unstabilized}.
\end{itemize}
\end{dfn}

Examples of reducible, weakly reducible and stabilized surfaces are presented in Fig.~\ref{fig1}.
\begin{figure}[!htb]
	\centering
\includegraphics[height=3.8cm]{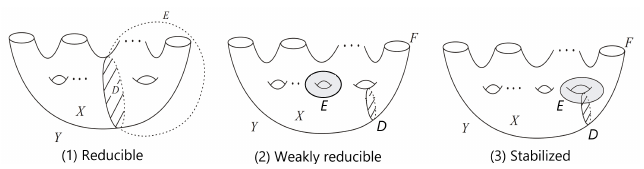}  	
	\caption{Reducible, weakly reducible and stabilized surfaces $F$.}
	\label{fig1}
\end{figure}

Let $H$ be a handlebody of genus $g(H)$, and $\ell$ a simple closed curve in $\partial H$. If there exists an essential disk $D$ of $H$ such that $\ell$ intersects $\partial D=m$ in a single point, we call $\ell$ a \textit{longitude} of $H$, $m$ a \textit{meridian curve} with respect to 
the longitude $\ell$, and $(\ell,m)$ a pair of longitude and meridian curves (or shortly, a \textit{LM-pair}). Let $T_1$ be a regular neighborhood of $\ell \cup m$ in $\partial H$. Then $T_1$ is a once-punctured torus in $\partial H$, and $\gamma = \partial T_1$ bounds a separating disk $D_{(\ell,m)}$ in $H$, which cuts $H$ into a solid torus $T$ (with a LM-pair $(\ell,m)$) and a handlebody of genus $g(H)-1$, see Fig.~\ref{fig2}. We call $D_{(\ell,m)}$ an \textit{associated disk} to the LM-pair $(\ell,m)$.

\begin{figure}[!htb]
	\centering
	\includegraphics[width=0.6\textwidth]{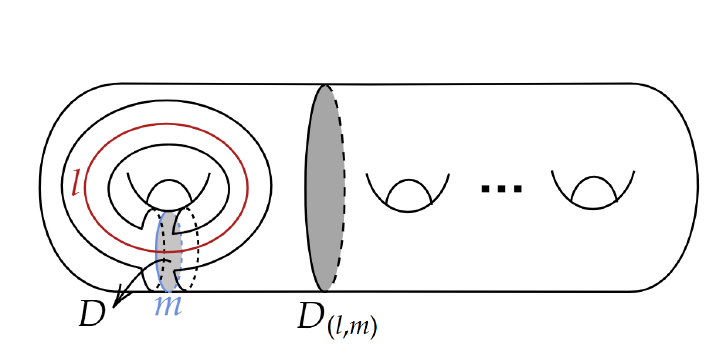}
	\caption{LM-pair $(\ell,m)$. }
	\label{fig2}
\end{figure}

\begin{prop} \label{pp2.3}
Let $M$ be a handlebody of positive genus, and $F$ a connected  incompressible surface in $M$ which cuts $M$ into a collection of handlebodies $\mathcal{H}$. Assume that $F$ is neither a disk nor an annulus, and $\Delta$ is a $\partial$-compressing disk of $F$ in $M$ such that $\Delta\cap F=\gamma$ is an essential arc in $F$ which is non-separating in $F$. Let $H$ be the component of $\mathcal{H}$ in which $\Delta$ is lying and $\beta=\partial\Delta$, $\alpha$ a simple closed curve in $F$ with $|\alpha\cap\gamma|=1$, and  $D_{(\alpha,\beta)}$ the associated disk to the LM-pair $(\alpha,\beta)$. After an isotopy of $D_{(\alpha,\beta)}$ in $H$,  $\partial D_{(\alpha,\beta)}$ is in general position with $F$ and $|F\cap\partial D_{(\alpha,\beta)}|$ is minimal, and denote $\delta= F\cap\partial D_{(\alpha,\beta)}$.
\begin{itemize}
\item[(1)] If $\alpha$ is boundary parallel in $F$, then $D_{(\alpha,\beta)}$ cuts $H$ into a solid torus $T$ with $\alpha\cup\beta\subset \partial T$ and a handlebody $H'$, moreover $\delta$ is a simple arc which cuts $F$ into an annulus $A\subset \partial T$ and a non-disk surface $S\subset \partial H'$;
\item[(2)] If $\alpha$ is non-boundary parallel in $F$, then $D_{(\alpha,\beta)}$ cuts $H$ into a solid torus $T$ with $\alpha\cup\beta\subset \partial T$ and a handlebody $H'$, moreover $\delta$ contains two simple arcs which cuts $F$ into an annulus $A\subset \partial T$ and a non-disk surface $S\subset \partial H'$ (possibly, non-connected).
In particular, $(\alpha\cup\beta) \cap S=\emptyset$.
\end{itemize}
\end{prop}

\begin{proof}
It follows directly from the definition of $D_{(\alpha,\beta)}$, see Fig.~\ref{fig3}.
\end{proof}

\begin{figure}[!htb]
	\centering
	\includegraphics[width=0.85\textwidth]{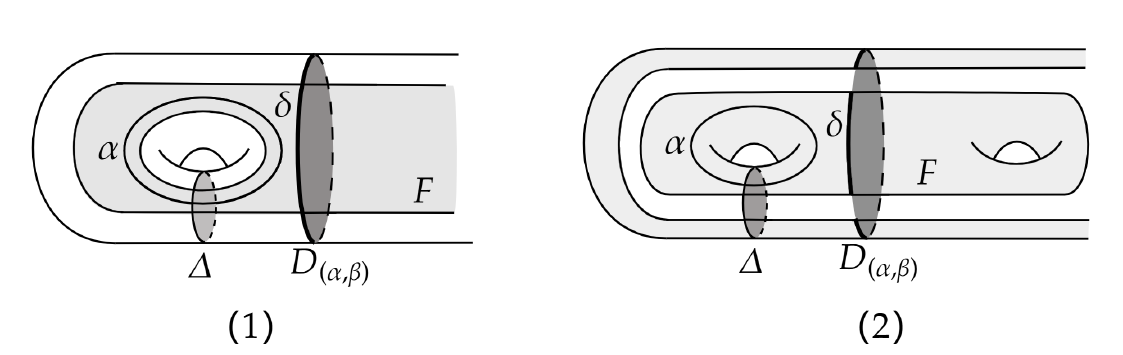}
	\caption{Case (1) and (2) in Proposition \ref{pp2.3}}
	\label{fig3}
\end{figure}

Let $H$ be a handlebody, and $\mathcal{J}$ a collection of pairwise disjoint simple closed curves in $\partial H$. Denote by $H(\mathcal{J})$ the 3-manifold obtained by attaching 2-handles to $H$ along the curves in~$\mathcal{J}$.

In the following, we present some known results which will be used in section~\ref{sec5}.

\begin{lem} [\cite{jaco1980lectures}] \label{lem2.5}
(1) A connected incompressible surface in a solid torus $T$ is either an essential disk or a boundary-parallel annulus in $T$ whose boundary curves are non-meridional.

(2) Let $H$ be a handlebody of genus $n \geq 2$. If $F$ is a connected, incompressible and $\partial$-incompressible surface in $H$, then $F$ is an essential disk in $H$.
\end{lem}

The following lemma was first proved by Przytycki~\cite{przytycki1983incompressibility} in 1983, and then it was generalized to the handle addition theorem by Jaco~\cite{jaco1984adding} in 1984.

\begin{lem} [\cite{przytycki1983incompressibility, jaco1984adding}] \label{lem2.6}
Let $H$ be a handlebody of genus $n\geq 1$, and $J$ a simple closed curve on $\partial H$. Suppose that $\partial H - J$ is incompressible in $H$. Then $H(J)$ has incompressible boundary, or $H(J)$ is a 3-ball. In the latter case, $H$ is a solid torus, and $J$ is a longitude for $H$.
\end{lem}

The following necessary and sufficient condition for a handle addition to a handlebody to be a handlebody was proved in~\cite[Th.~2.6]{lei2020necessary}.

\begin{prop} [\cite{lei2020necessary}] \label{pp2.6}
Let $H$ be a handlebody of genus $n$, $J$ a simple closed curve in $\subset \partial H$. Then $H(J)$ is a handlebody if and only if $J$ is a longitude for $H$.
\end{prop}

\begin{prop} [\cite{lei1994some}] \label{pp2.7}
Let $M$ be a 3-manifold, and $T$ a solid torus, $M'=M\cup_A T$ an amalgamation of $M$ and $T$ along annuli $A$. If the core curve of $A$ is a longitude of $T$, then $M' \cong M$.
\end{prop}

The following uniqueness theorem of Heegaard splittings of $S^3$ is due to Waldhausen~\cite{waldhausen1968heegaard}.

\begin{thm} [\cite{waldhausen1968heegaard}]  \label{thm2.7}
Any positive genus Heegaard splitting of $S^3$ is stabilized.
\end{thm}


\begin{lem} \label{lem2.9}
Let $M$ be a compact orientable 3-manifold with boundary. Then $M$ contains a properly embedded, 2-sided, incompressible surface $F$ with $\partial F\ne \emptyset$.
\end{lem}

The proof of Lemma \ref{lem2.9} can be found in \cite[Theorem~III.10]{jaco1980lectures} and \cite[Lemma~6.8]{hempel20043}.

The notion of a primitive set of disjoint simple closed curves in the boundary of a handlebody was introduced by Gordon in~\cite{gordon1987primitive}.

\begin{dfn} \label{d2.8}
Let $\mathcal{J}=\{J_1,\ldots,J_m\}$ be a collection of simple closed curves in the boundary of a handlebody $H$ of genus $n$.
\begin{itemize}
\item[(1)] If $\{[J_1],\ldots,[J_m]\}\subset \pi_1(H)$ (after some conjugation) can be extended to a basis of $\pi_1(H)$, we say that $\mathcal{J}$ is \textit{primitive} in $H$.
\item[(2)] If the curves in $\mathcal{J}$ are pairwise disjoint, and $H(\mathcal{J})$ is a handlebody (of genus $n-m$), we say that $\mathcal{J}$ is \textit{quasi-primitive} in $H$.
\end{itemize}

Let $\mathcal{H}$ be a collection of finite number of handlebodies, and $\mathcal{J}$ a collection of simple closed curves in $\partial\mathcal{H}=\bigcup_{H_i\in\mathcal{H}}\partial H_i$. We say that $\mathcal{J}$ is \textit{primitive} in $\mathcal{H}$ if denote by $\mathcal{J}_i$ the subset of the curves in $\mathcal{J}$ which are lying in $\partial H_i$, $\mathcal{J}_i$ is primitive in $H_i$ for each non-empty $\mathcal{J}_i$.
\end{dfn}

In definition \ref{d2.8} it is clear, if the curves in $\mathcal{J}$ are pairwise disjoint and  $\mathcal{J}$ is primitive, then $\mathcal{J}$ is quasi-primitive. The converse is generally not true except for $m=1$. But it was shown by Gordon in~\cite[Th.~1]{gordon1987primitive} that if all subsets $\mathcal{J}'$ of $\mathcal{J}$ are quasi-primitive in $H$, then $\mathcal{J}$ is primitive in $H$.

The following theorem~\cite[Th.~3.2]{lei2020necessary} gives a characteristics of an $H'$-splitting $X\cup_F Y$ for a handlebody $H$, where $F$ is incompressible in $H$.

\begin{thm} [\cite{lei2020necessary}] \label{t2.9}
Let $X\cup_F Y$ be an $H'$-splitting for a 3-manifold $M$, where \linebreak $g(X), g(Y) \geq 2$. Suppose that $F$ is incompressible in both $X$ and $Y$. Then $M$ is a handlebody if and only if there exists a basis curve set $\mathcal{J}=\{J_1,\ldots,J_m\}$ of $\pi_1(F)$ with a partition $(\mathcal{J}_1,~\mathcal{J}_2)$ of $\mathcal{J}$ such that $\mathcal{J}_1$ is primitive in $X$ and $\mathcal{J}_2$ is primitive in $Y$.
\end{thm}

For convenience, we also say that a basic curve set $\mathcal{J}$ in Theorem~\ref{t2.9} is \textit{primitive} in $\{X,Y\}$. The following is a property of a collection of simple closed curves on the boundary of a handlebody, see~\cite[Th.~2.7]{lei2020necessary}.

\begin{lem} [\cite{lei2020necessary}] \label{lem14}
Let $H$ be a handlebody of genus $n$, and ${\mathcal J}=\{J_1,\cdots,J_p\}$ ($p<n$) a collection of simple closed curves on $\partial H$. Suppose that there exists a collection ${\mathcal D}=\{D_1,\ldots,D_p\}$ of pairwise disjoint disks properly embedded in $H$ such that $|J_i\cap \partial D_i|=1$ for $1\leq i\leq p$, and $|\partial D_i\cap J_j|=0$ for $1\leq i< j\leq p$. Then ${\mathcal J}$ is  primitive.
\end{lem}


\begin{prop} \label{pp3.1}
Suppose that $M$ is irreducible. If $F$ is reducible, then $F$ is stabilized.
\end{prop}

\begin{proof}
By assumption, there exist compression disks $D_1$ and $D_2$ of $F$ in $M$ which are lying in the two sides of $F$, respectively, such that $\partial D_1=\partial D_2=\alpha$. Since $M$ is irreducible, the 2-sphere $S= D_1\cup D_2$ bounds a 3-ball $D^3$ in $M$. Since $S\cap F=\alpha$ and $S$ is separating in $M$, $\alpha$ is separating in $F$, and $D^3\cap F=F_1$ splits $D^3$ into two handlebodies of the same positive genus. Thus, $F_1$ can be extended to a Heegaard surface $\widehat{F_1}$ of positive genus in $\widehat{D^3}=S^3$. By Theorem~\ref{thm2.7},  any positive genus Heegaard splitting of $S^3$ is stabilized. Hence $\widehat{F_1}$ is stabilized in $S^3$, then $F$ is stabilized in $M$.
\end{proof}

\begin{prop} \label{pp3.2}
Let $M$ be an amalgamation of $\mathcal{M}$ along $F_1$ and $F_2$, where $\mathcal{M}=\{M\}$ or $\mathcal M = \{M_1,M_2\}$, and in the later case, $F_i\subset \partial M_i$, $i=1,2$. Set $F=F_1=F_2$ in $M$. Then $F$ is incompressible in $M$ if and only if $\{F_1,F_2\}$ is incompressible in $\mathcal{M}$.
\end{prop}

\begin{proof}
The necessity is clear. Now assume that $\{F_1,F_2\}$ is incompressible in $\mathcal M$. If $F$ is compressible in $M$, there exists a disk $D$ in $M$, such that $D\cap F=\partial D$, and $\partial D$ is essential in $F$. Thus $D$ is a compression disk for either $F_1$ or $F_2$ in $\mathcal{M}$, contradicting to the assumption.
\end{proof}

In the following, we always assume that $M$ is an amalgamation of a collection $\mathcal{H}$ of handlebodies along $F_1$ and $F_2$, where $F=F_1=F_2$ is a compact connected orientable surface in $M$, and $\mathcal{H}=\{H_1,H_2\}$ or $\mathcal H = \{H\}$, depending on that $F$ is separating or non-separating in $M$, except for the other stated. We will focus on that $M$ is a handlebody, and split discussions into two cases: when the amalgamated surface $F$ is incompressible or compressible in $M$.

\section{Incompressible surfaces in a handlebody} \label{sec3}

In this section we consider the case when an amalgamated surface $F$ is incompressible.

\begin{thm}\label{tt3.3}
Let $M$ be a compact connected orientable 3-manifold. Then $M$ is a handlebody if and only if any 2-sided incompressible surface $F$ in $M$ can be $\partial$-compressed into a collection $S$ of essential disks in $M$, such that each component of $M\setminus F$ is a handlebody.
\end{thm}

\begin{proof}	
First assume that $M$ is a handlebody of genus $n$. If $n=0$, then $M$ is a 3-ball and there is no incompressible surface in $M$. For $n\geq 1$, let $F$ be an orientable incompressible surface in $M$. Since $M$ contains no closed incompressible surfaces, each component of $F$ is a surface with non-empty boundary. Choose a complete system $\mathcal{D}=\{D_1,\cdots,D_n\}$ of disks in $M$ such that $\mathcal{D}$ is in general position with $F$ and $D=\bigcup_{i=1}^n D_i$ is such number of intersections $|D\cap F|$ is minimal over all such complete systems of disks and all isotopies of $F$ in $M$.

{\bf Claim 1.} \textit{$D\cap F$ has no circle component.}

Otherwise, let $\alpha$ be an innermost circle component of $D_i \cap F$ for some $i$. Then $\alpha$ bounds a disk $\Delta$ in $D_i$ with $F\cap\operatorname{int}(\Delta)=\emptyset$. Since $F$ is incompressible, $\alpha$ bounds a disk $\Delta'$ in $F$. Then $\Delta\cup\Delta'$ is a 2-sphere in $M$. By the irreducibility of $M$, $\Delta\cup\Delta'$ bounds a 3-ball in $M$, see Fig.~\ref{fig4}. By using the 3-ball, we can isotope $F$ to $F'$ by pushing $\Delta'$ to cross $\Delta$ in $M$ such that $|D\cap F'|<|D\cap F|$, a contradiction to the minimality of $|D\cap F|$.
\begin{figure}[!htb]
	\centering
	\includegraphics[width=0.3\textwidth]{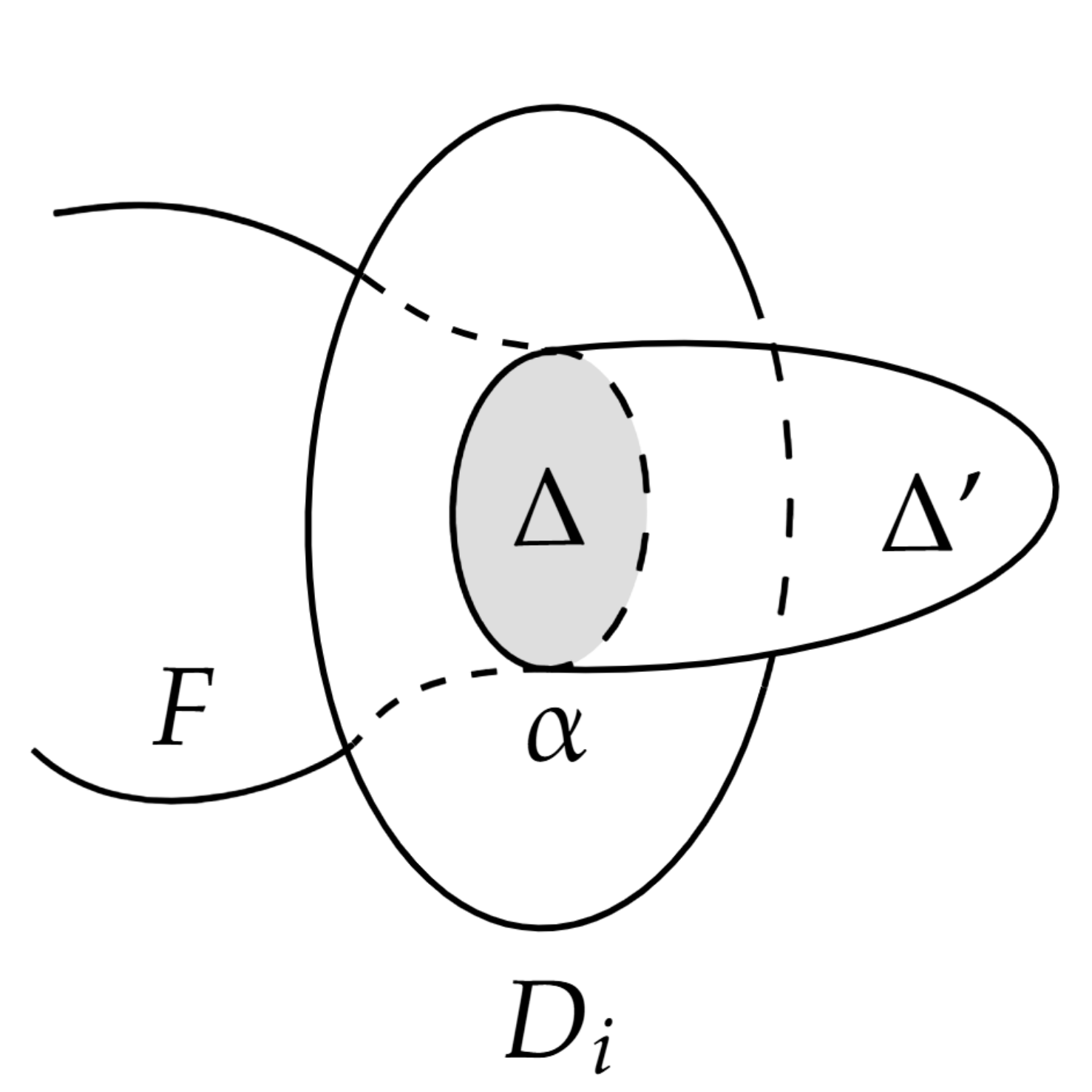}
    \caption{The 2-sphere $\Delta\cup\Delta'$.}
    \label{fig4}
\end{figure}

{\bf Claim 2.} \textit{If $D\cap F\ne\emptyset$, then each component of $D\cap F$ is an essential arc in $F$.}

Otherwise, let $\beta$ be a component of $D\cap F$ (say, $D_j\cap F$ for some $j$) which is outermost on $F$, i.e., $\beta$ cuts out of a disk $\Delta$ from $F$ with $D\cap\operatorname{int}(\Delta)=\emptyset$. Then $\beta$ cuts $D_j$ into two disks $D_j^1$ and $D_j^2$. Set $D_j'=\Delta\cup D_j^1$ and  $D_j''=\Delta\cup D_j^2$, see Fig.~\ref{fig5}.
\begin{figure}[!htb]
	\centering
	\includegraphics[width=1.\textwidth]{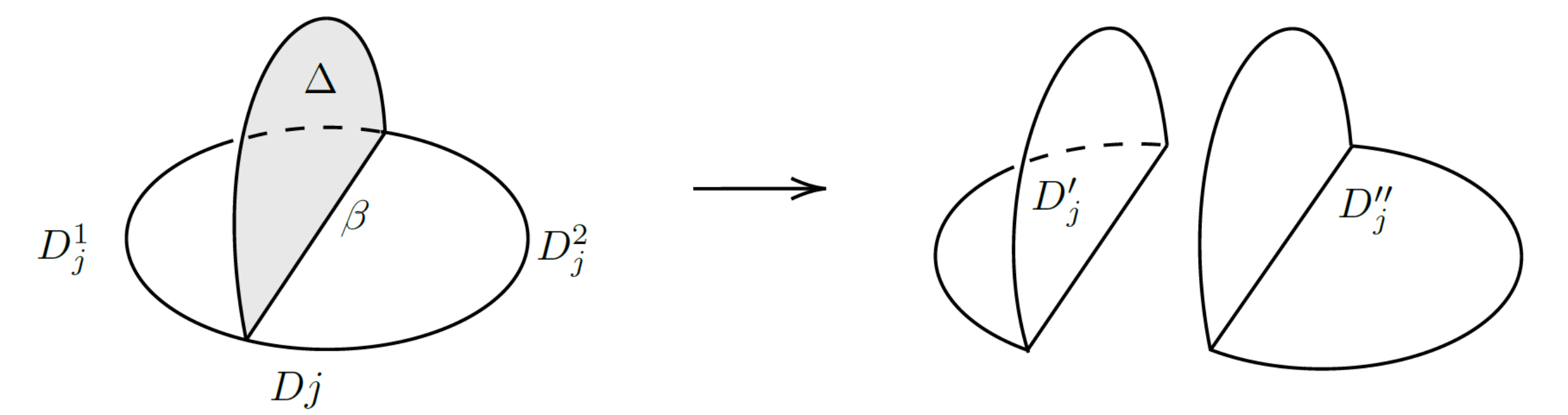}
    \caption{$D_j'$ and $D_j''$.}
    \label{fig5}
\end{figure}

Denote $\mathcal{D}'=(\mathcal{D}\setminus D_j)\cup D_j'$ and $\mathcal{D}''=(\mathcal{D}\setminus D_j)\cup D_j''$.
Then at least one of $\mathcal{D}'$ and $\mathcal{D}''$, say, $\mathcal{D}'$, is a complete system of disks for $M$. Let us denote  by $D'$ the union of all disks in $\mathcal{D}'$. Then after a small isotopy of $\mathcal{D}'$, we have $|D'\cap F|<|D\cap F|$, again a contradiction to the minimality of $|D\cap F|$.
Claim 2 implies that $D$ is disjoint from all the disk components of $F$.

Let $S=S_{g,b}$ be the compact connected surface with genus $g$ and $b$ boundary components. Denote
\begin{equation}
C(S_{g,b}) = -\chi(S_{g,b}) +1 = 2g + b - 1. \label{eqn3.1}
\end{equation}
Let $\alpha$ be an essential arc properly embedded in $S$ ($b>0$). Denote by $S'$ the surface obtained by cutting $S$ open along $\alpha$. It is clear that $C(S')=C(S)-1$. If $F$ has $k$ components $F_1,\ldots,F_k$, set $\chi(F)=\sum_{i=1}^k\chi(F_i)$, and call $C(F) = \sum_{i=1}^k C(F_i)$ the complexity of $F$. It is clear that $C(F)\geq 0$, and $C(F)=0$ if and only if each component $F_i$ of $F$ is a disk, $1\leq i\leq k$. Thus, the conclusion holds when $C(F)=0$. We will prove the necessity of the theorem by induction on $C(F)$. Assume that the conclusion holds for all orientable incompressible surface $F'$ in $M$ with $C(F')=p\geq 0$.

Let $F$ be an orientable incompressible surface in $M$ with $C(F)=p+1$.  As well as above, let us choose a complete system $\mathcal{D}=\{D_1,\cdots,D_n\}$ of disks in $M$ such that $\mathcal{D}$ is in general position with $F$ and $D=\bigcup_{i=1}^n D_i$ is such number of intersections $|D\cap F|$ is minimal over all such complete systems of disks and all isotopies of $F$ in $M$. Then $D\cap F\ne \emptyset$. Let $\gamma$ be an arc component of $D\cap F$ which is outermost in $D$. Then $\gamma$ cuts out of a disk $E$ from $D$ with $F\cap\operatorname{int}(E)=\emptyset$. By Claims 1 and 2, $\gamma$ is essential in $F$, so $E$ is a $\partial$-compression disk for $F$. Let $F^*$ be the surface obtained by doing a $\partial$-compression in $M$ on $F$ along $E$. Then $C(F^*)=C(F)-1=p$. By the inductive assumption, $F^*$ can be $\partial$-compressed into a collection $\mathcal{D}^*$ of pairwise disjoint essential disks in $M$, such that each component of $M\setminus \mathcal{D}^*$ is a handlebody.

This completes the proof of necessity of the theorem.

Now assume that $M$ is a 3-manifold with $\partial M \neq \emptyset$, which satisfies the given condition. By Lemma~\ref{lem2.9}, there exists a 2-sided incompressible surface $F$ in $M$ with $\partial F\ne \emptyset$. Then $F$ can be $\partial$-compressed into a collection $\mathcal{D}^*$ of pairwise disjoint essential disks in $M$, such that each component of $M\setminus \mathcal{D}^*$ is a handlebody. It follows that $M$ is a handlebody.
\end{proof}

Let $H$ be a handlebody, and $F$ an 2-sided incompressible surface (maybe non-con\-nec\-ted) in $H$. Since $H$ contains no closed incompressible surfaces, each component of $F$ is a surface properly embedded in $H$ with non-empty boundary. Let us denote $\mathcal{H}^0=\{H\}$, $F_0=F$, and $\mathcal{H}^1=H\setminus F_0=\{H_1,\ldots, H_q\}$. Then $H=\mathcal{H}^1_{F_0}$. If $F_0$ has a non-disk component, then as shown in the proof of Theorem~\ref{tt3.3}, there exists a $\partial$-compression disk, say $E_0$, of $F_0$ in $H$, such that $E_0$ is an essential disk of $H_{j_1}$ for some $1\leq j_1\leq q$, and $E_0\cap F_0= \ell_1$ an essential arc in $F_0$. Let us denote $\ell'=\partial E_0\cap \partial H=\overline{\partial E_0 -\ell_1}$. Let $N_1=E_0\times [0,1]$ be a regular neighbourhood of $E_0$ in $H_{j_1}$ such that $N_1\cap \partial H = \ell'\times [0,1]$ and $N_1\cap F_0 = \ell_1\times I$. Consider
$$
F_1=(F_0 - \ell_1\times [0,1]) \cup E_0 \times \{0,1\},
$$
see Fig.~\ref{fig6}.
\begin{figure}[!htb]
	\centering
	\includegraphics[width=1.0\textwidth]{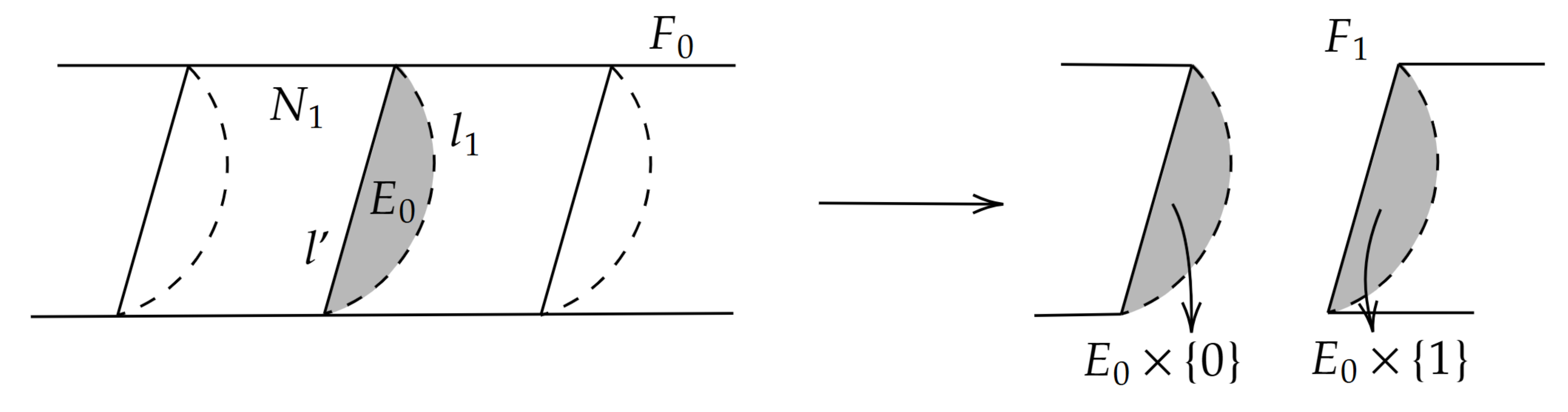}
    \caption{$\partial$-compression on $F_0$ along $E_0$.}
    \label{fig6}
\end{figure}

Then $F_1$ is the surface obtained by doing a $\partial$-compression in $H$ along $E_0$. Denote $\mathcal{H}^2=H\setminus F_1$. Then $H = {\mathcal H}^1_{F_0} =  {\mathcal H}^2_{F_1}$. If $F_1$ has a non-disk component, we may continue the process to construct $F_2$ and $\mathcal{H}^3$. Finally, we get $F_{k}$ and $\mathcal{H}^{k+1}$, where $F_{k}$ is a collection of essential disks in $H$, and $\mathcal{H}^{k+1}$  is a collection of handlebodies and $H={\mathcal H}^{k+1}_{F_k}$. We call
\begin{equation}
({\mathcal{H}^0},F_0,E_0)\to({\mathcal{H}^1},F_1,E_1)\to \cdots\to ({\mathcal{H}^k},F_k,\emptyset)\to \mathcal{H}^{k+1} \label{eqn3.2}
\end{equation}
a \textit{$\partial$-compression hierarchy} of $H$, and call
$$
({\mathcal{H}^0}, F_0, E_0) \to ({\mathcal{H}^1}, F_1, E_1)\to \cdots \to ({\mathcal{H}^p}, F_p, E_p)
$$
a \textit{partial $\partial$-compression hierarchy} of $H$ for $p<k$. Note that, by our choice of $E_i$, $1\leq i \leq k$, $E_i$ is an essential disk in a handlebody in $\mathcal{H}$, and they are pairwise disjoint and non-parallel; $E_i \cap F_i$ is a single essential arc in $F_i$ with $E_i\cap F_j = \emptyset$ for $0\leq i < j \leq k$; moreover, each $\mathcal{H}^{i}$ is a collection of handlebodies, $0 \leq i \leq k+1$.

Theorem~\ref{tt3.3} can be reformalated as follows.

\vskip 2mm
\noindent
{\bf Theorem 3.1$'$} \textit{Let $M$ be a compact connected orientable 3-manifold. Then $M$ is a handlebody if and only if for any orientable incompressible surface $F$ in $M$, $M$ can be splitted into a collection $\mathcal{H}$ of handlebodies along $F$, and there exists a $\partial$-compression hierarchy~(\ref{eqn3.2})
for $M$.}

\vskip 2mm
As a direct consequence of Theorem~3.1$'$, we obtain a new proof of the following known result by Schultens, see~\cite{J.Schu}.

\begin{cor} \label{cc3.4}
Suppose that $H$ is a handlebody of positive genus, and $F$ is an incompressible surface in $H$. Then each component of $H\setminus F$ is a handlebody.
\end{cor}

\begin{proof}
By Theorem 3.1$'$, there exists a $\partial$-compression hierarchy~(\ref{eqn3.2}) for $H$, and $\mathcal{H}^{k+1}$ is a collection of handlebodies. For each $i$, $1\leq i\leq k$, $\mathcal{H}^i$ is an amalgamation of $\mathcal{H}^{i+1}$ along $E_i$, so $\mathcal{H}^i$ is a collection of handlebodies. Thus $M\setminus F=\mathcal{H}^1$ is a collection of handlebodies.
\end{proof}

Let $M$ be a 3-manifold with non-empty boundary $\partial M$, and $F$ a surface properly embedded in $M$ which has no trivial disk components. Let $\alpha$ be a simple arc in $\partial M$ with $\alpha\cap\partial F=\partial\alpha=\{A,B\}$, and $\Delta$ is a disk in $F$ such that
\begin{itemize}
\item $\Delta\cap \partial F=\partial\Delta\cap \partial F=\beta$ is an arc in $\partial\Delta$,
\item $\beta\cap\alpha={A}\in\text{Int}(\beta)$,
\item and $\overline{\partial\Delta-\beta}=\gamma$ is an arc in $\partial\Delta$ with $\partial\beta\cap\partial\gamma=\partial\beta=\partial\gamma$ and $\beta \cup \gamma =\partial\Delta$.
\end{itemize}
Let $N=\alpha\times \Delta$ be a cylinder embedded in $M$, such that
\begin{itemize}
\item $N\cap \partial M=\alpha\times\beta$,
\item $N\cap F=\partial\alpha\times\Delta=\Delta\cup\Delta'$, where $\Delta'$ is a disk in $F$, therefore $F\cap (\alpha\times \gamma)$ $=(\partial\alpha)\cap F=\{A\}\times \gamma\cup \{B\}\times\gamma$,
\item $\Delta'\cap\partial M=\{B\}\times\beta$, which is a arc in $\Delta'$,
\end{itemize}
see Fig.~\ref{fig7}. Set $F(\alpha)=\overline{F\setminus (\Delta\cup\Delta')}\bigcup \alpha\times\gamma$. Then $F(\alpha)$ is a surface properly embedded in $M$. We call $F(\alpha)$ a {\em band sum} of $F$ along $\alpha$.
\begin{figure}[!htb]
	\centering
	\includegraphics[width=0.5\textwidth]{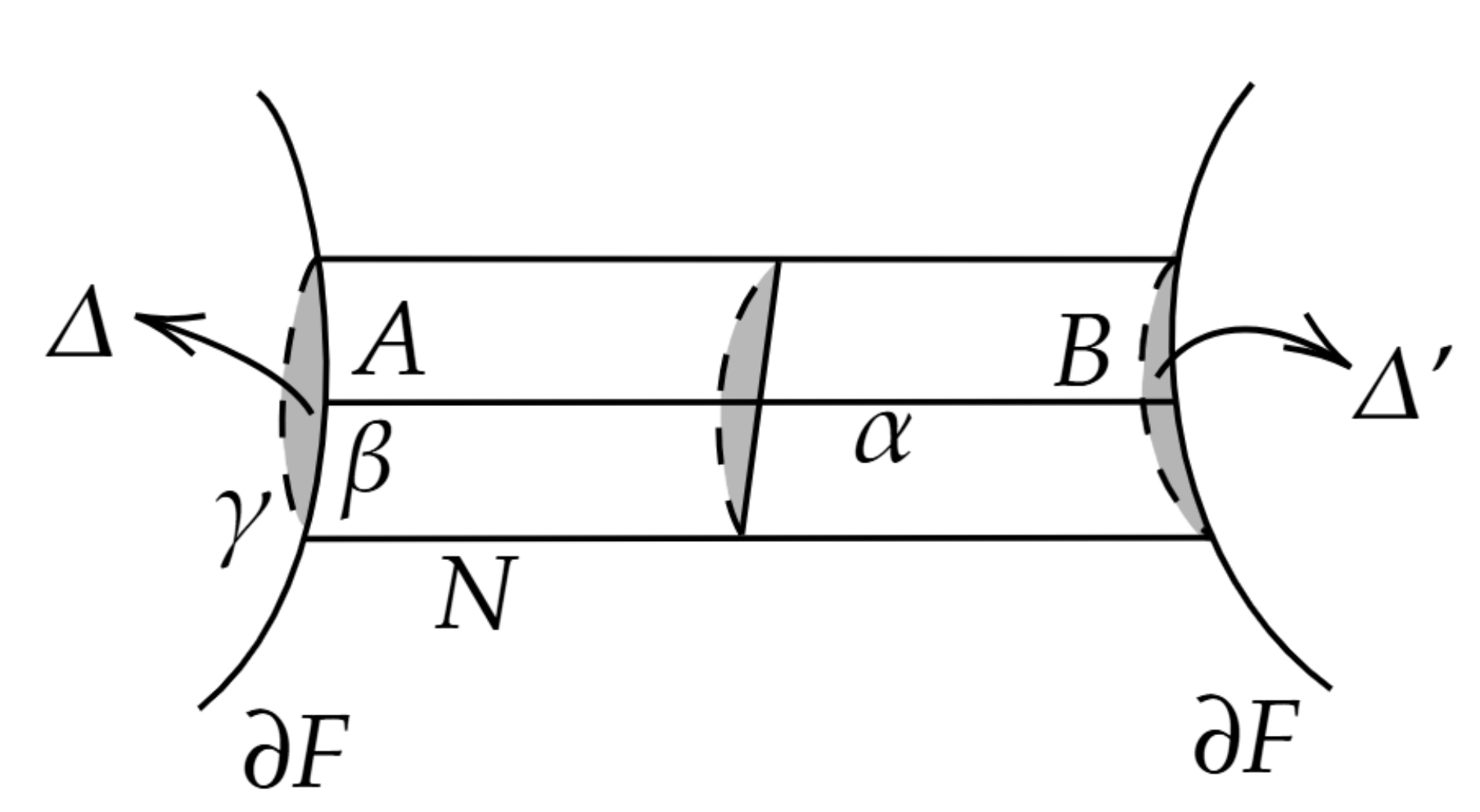}
	\caption{$F(\alpha)$ a \textit{band sum} of $F$ along $\alpha$.} \label{fig7}
\end{figure}

It is clear that $\Delta$ is a $\partial$-compressing disk of $F(\alpha)$, and the surface obtained by doing a $\partial$-compression on $F(\alpha)$ along $\Delta$ is isotopic to $F$ in $M$. We call $\alpha$ the arc \textit{associated} to $\Delta$ and denote $\alpha$ by $\alpha_\Delta$. In particular,
\begin{itemize}
\item if $\partial\alpha$ is lying in the boundaries of different components $F_1$ and $F_2$ of $F$, and $F=F_1\cup F_2$, we also say that $F(\alpha)$ is a {\em band sum} of $F_1$ and $F_2$ along $\alpha$;
\item if $F$ is connected, we say that $F(\alpha)$ is a \textit{self-band sum} of $F$ along $\alpha$.
\end{itemize}

\begin{cor} \label{cc3.6}
Let $M$ be a handlebody of positive genus. Then for any incompressible surface $F$ in $M$, there exists a collection $S_0$ of pairwise disjoint essential disks in $M$, and a sequence of incompressible surfaces $S_0,S_1,\ldots,S_k$ in $M$, such that $S_{i+1}$ is a band sum of $S_i$ along $\alpha_i$, $0\leq i\leq k-1$, and $S_k=F$.
\end{cor}

\begin{proof}
By Theorem 3.3$'$, there exists a $\partial$-compression hierarchy $$({\mathcal{H}^0},F_0,E_0)\to({\mathcal{H}^1},F_1,E_1)\to \cdots\to ({\mathcal{H}^k},F_k,\emptyset)\to \mathcal{H}^{k+1}$$
for $M$, where $S_0=F_k$ is a collection of pairwise disjoint essential disks in $M$. Set $S_i=F_{k-i}$, $1\leq i\leq k-1$, and denote the arc associated to $E_i$ by $\alpha_{k-i-1}$, $0\leq i\leq k-1$. The conclusion follows directly.
\end{proof}

\section{Compressible surfaces in a handlebody} \label{sec4}


Suppose that surface $F$ is compressible in a 3-manifold $M$, and no component of $F$ is a disk or 2-sphere. Then there exists a collection $\Sigma$ of pairwise disjoint non-parallel compression disks in $M$ such that if $\widehat{F}$ is the surface obtained by compressing $F$ along $\Sigma$ in $M$, then each component of $\widehat{F}$ is either an incompressible surface, or a 2-sphere in $M$. We call $\Sigma$ a {\em maximal system} of compression disks for $F$. If $\Sigma$ contains no 2-sphere components, $\Sigma$ is called a \textit{good maximal system} of compression disks for $F$.

\begin{thm} \label{tt3.6}
Assume that surface $F$ is compressible in a 3-manifold $M$. Then either $F$ is reducible, or there exists a good maximal system $\Sigma$ of compression disks for $F$, such that each component of $\widehat{F}$, the surface obtained by compressing $F$ along $\Sigma$ in $M$, is an incompressible surface with non-empty boundary. Moreover, if $F$ is non-separating in $M$, at least one component of $\widehat{F}$ is non-separating in $M$.
\end{thm}

\begin{proof}
By assumption, $F$ is compressible in $M$. Choose a maximal system $\Sigma$ of compression disks for $F$ such that cardinality $|\Sigma|$ is minimal over all maximal systems of compression disks for $F$. Then each component of $\widehat{F}$ is either an incompressible surface with non-empty boundary, or a 2-sphere in $M$. In fact, if a component $S$ of $\widehat{F}$ is closed with positive genus, since $M$ is a handlebody, $S$ is compressible in $M$, contradicting to the assumption that $|\Sigma|$ is minimal.

If $F$ is reducible, the conclusion holds. In the following, we assume that $F$ is irreducible.

Denote by $\mathcal{D}=\{D_1,\ldots,D_m\}$ the set of disks in $\Sigma$ which are lying in one side of $F$, and $\mathcal{E} =\{E_1,\ldots,E_n\}$ the set of disks in $\Sigma$ which are lying in the other side of $F$. Let $S$ be a 2-sphere component of $\widehat{F}$. $S$ is obtained by capping a planar sub-surface $P$ of $F$ with some cutting sections of $\Sigma$. We split our discussion in two cases.

{\bf Case 1}. Suppose $\partial P$ consists of the cutting sections of the disks only in $\mathcal{D}$ (or $\mathcal{E}$). Assume that the case of $\mathcal D$ happens. Denote by $D_i'$ and $D_i''$ the two cutting sections of $D_i$ in $\mathcal{D}$. There must be some component, say $\partial D_i'$, of $\partial P$, such that $\partial D_i''$ is not a component of $\partial P$. Otherwise, $F$ is closed, contradicting to the assumption. It is easy to see that $\Sigma'=\Sigma\setminus \{D_i\}$ is still a  maximal system of compression disks for $F$ with $|\Sigma'|=|\Sigma|-1$, contradicting to the minimality of $|\Sigma|$. The case of $\mathcal E$ can be considered analogously.

{\bf Case 2}. Suppose $\partial P$ consists of some cutting sections of the disks in both $\mathcal{D}$ and $\mathcal{E}$. Let $\alpha$ be a simple closed curve on $P$ such that $\alpha$ cuts $P$ into two planar surfaces $P_1$ and $P_2$ such that $\partial P_1$ consists of $\alpha$ and some cutting sections of the disks in $\mathcal{D}$ and $\partial P_2$ consists of $\alpha$ and some cutting sections of the disks in $\mathcal{E}$, see Fig.~\ref{fig8}.
\begin{figure}[!htb]
	\centering
	\includegraphics[width=0.8\textwidth]{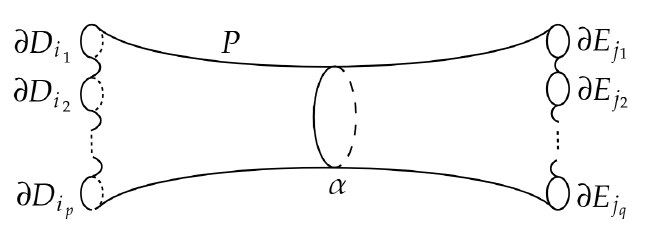}
    \caption{A simple closed curve $\alpha$ in $P$.}
    \label{fig8}
\end{figure}
Thus $\alpha$ is essential in $F$ and bounds disks in both sides of $F$, hence $F$ is reducible in $M$, contradicting to the assumption.
Therefore, $\Sigma=\mathcal{D}\cup\mathcal{E}$ is a good maximal system $\Sigma$ of compression disks for $F$.

If each component of $\widehat{F}$ is non-separating in $M$, then $F$ is separating in $M$. Thus, if $F$ is non-separating in $M$, at least one component of $\widehat{F}$ is non-separating in $M$.
This completes the proof.
\end{proof}

\begin{cor} \label{cc3.7}
Suppose that $M$ is a handlebody, and $F$ is compressible in $M$. Then either $F$ is stabilized, or there exists a good maximal system $\Sigma$ of compression disks for $F$, such that each component of $\widehat{F}$, the surface obtained by compressing $F$ along $\Sigma$ in $M$, is an incompressible surface with non-empty boundary. Moreover, if $F$ is non-separating in $M$, at least one component of $\widehat{F}$ is non-separating in $M$.
\end{cor}

\begin{proof}
Note that $M$ is irreducible. If $F$ is reducible, then $F$ is stabilized by Proposition~\ref{pp3.1}. Otherwise the conclusion follows from Theorem~\ref{tt3.3}.
\end{proof}

Let $M$ and $F$ be as in the Corollary~\ref{cc3.6}, and assume that $F$ is irreducible in $M$. By Corollary \ref{cc3.6}, there exists a maximal system $\Sigma$ of compression disks for $F$, such that each component of $\widehat{F}$, the surface obtained by compressing $F$ along $\Sigma$ in $M$, is an incompressible surface with non-empty boundary. Note that for $\Sigma=\mathcal{D}\cup\mathcal{E}$, one of $\mathcal{D}$ and $\mathcal{E}$ may be empty. If both $\mathcal{D}$ and $\mathcal{E}$ are non-empty, then $F$ is weakly reducible. Using above notations, we have the following result.

\begin{cor} \label{cc3.8}
Suppose that $M$ is a handlebody, and $F$ is compressible in $M$. Then either $F$ is stabilized, or there exists a good maximal system $\Sigma=\mathcal{D}\cup\mathcal{E}$ of compression disks for $F$, and there exist a division $\{\mathcal{D}_1, \cdots, \mathcal{D}_r\}$ and a division $\{\mathcal{E}_1, \ldots, \mathcal{E}_s\}$, such that $\mathcal{D}_i$ and $\mathcal{E}_j$ are quasi-primitive in $M$, $1\leq i\leq r$, $1\leq j\leq s$.
\end{cor}

\begin{proof}
Again note that $M$ is irreducible. As well as before, if $F$ is reducible, then $F$ is stabilized by Proposition~\ref{pp3.1}. In the following, assume that $F$ is irreducible. It follows from Corollary~\ref{cc3.6} that there exists a good maximal system $\Sigma=\mathcal{D}\cup\mathcal{E}$ of compression disks for $F$, such that each component of $\widehat{F}$, the surface obtained by compressing $F$ along $\Sigma$ in $M$, is an incompressible surface with non-empty boundary.

Denote $\widehat{F}=\{F_1,\cdots,F_q\}$. Then $\widehat{F}$ cuts $M$ into a collection $\widehat{\mathcal{H}}$ of 3-manifolds, say $\widehat{\mathcal{H}}=\{X_1,\cdots,X_p\}$, and $M$ is an amalgamation of $\widehat{\mathcal{H}}$ along $\widehat{F}$. By Corollary~\ref{cc3.4}, each $X_i$ is a handlebody, $1\leq i\leq p$. A good maximal system $\Sigma$ cuts $M$ into a collection of handlebodies $\{Y_1,\ldots,Y_p\}$, and each $X_i$ is obtained by adding some 2-handles along a subfamily $\mathcal{J}_i$ (maybe empty) of some curves in $\partial \Sigma$ to, say $Y_i$, $1\leq i\leq p$. Note that the disks in $\mathcal{D}$ and the disks in $\mathcal{E}$ are lying in the different sides of $F$, hence $\mathcal{J}_i$ cannot contains the curves in both $\partial\mathcal{D}$ and $\partial\mathcal{E}$. Assume that $\emptyset\ne\mathcal{J}_i\subset \partial\mathcal{D}$ or $\partial\mathcal{E}$. Since $X_i$ is a handlebody, it follows that $\mathcal{J}_i$ is quasi-primitive in $Y_i$ (therefore in $M$). It is clear that $\mathcal{J}_i\cap \mathcal{J}_j=\emptyset$ for $i\ne j$. The conclusion the follows.
\end{proof}

\section{Amalgamation along an annulus} \label{sec5}

Let $M$ be a handlebody. Assume that $M$ is an amalgamation of a collection $\mathcal{H}$ of handlebodies along $F'$ and $F''$, where $F=F'=F''$ is a compact connected orientable surface in $M$, and $\mathcal{H}=\{H_1,H_2\}$ or $\{H\}$, depending on that $F$ is separating or non-separating in $M$. Suppose that $F$ is irreducible in $M$. It follows from Corollary \ref{cc3.6} that there exists a good maximal system $\Sigma$ of compression disks for $F$, such that each component of $\widehat{F}$, the surface obtained by compressing $F$ along $\Sigma$ in $M$, is an incompressible surface with non-empty boundary. Suppose that $\widehat{F}=\{F_1,\ldots,F_q\}$, and $\mathcal{H}\setminus   \Sigma=\widehat{\mathcal{H}}=\{X_1,\ldots,X_p\}$.  Then $\widehat{\mathcal{H}}$ is a collection of handlebodies, and $M$ is an amalgamation of $\widehat{\mathcal{H}}$ along $\widehat{F}$. If $F$ is non-separating in $M$, then at least one component of $\widehat{F}$ is non-separating in $M$.

Denote by $X\#_\partial Y$ the boundary connected sum of 3-manifolds $X$ and $Y$. It is clear that $X\#_\partial Y$ is a handlebody if and only if both $X$ and $Y$ are handlebodies. For an $H'$-splitting $X\cup_A Y$ for a 3-manifold $M$, where $A$ is an annulus, it is known~\cite{lei1994some} that $M$ is a handlebody if and only if the core curve of $A$ is a longitude of either $X$ or $Y$. In particular, $H(J)$ is a handlebody if and only if $J$ is a longitude of $H$ (or equivalently, $J$ is primitive in $H$).

A characteristic for an amalgamation of two handlebodies along connected incompressible surfaces to be a handlebody is given in Theorem~\ref{t2.9}. We will show below that a similar result holds for a self-amalgamation of a handlebody along incompressible surfaces. Firstly we consider the case when the amalgamated surface is an annulus.

Let $H$ be a handlebody of positive genus, and $A$ an annulus on $\partial H$. We say that $A$ is {\em separating} if the core curve of $A$ is separating in $\partial H$. Otherwise, $A$ is {\em non-separating}. Let $M$ be a self-amalgamation of $H$ along annuli $A'$ and $A''$. It is clear that $\partial H$ has a single component if and only if at least one of $A'$ and $A''$ is non-separating.
Let $\mathcal{H}$ be a collection of handlebodies. If $\ell$ is a longitude of some handlebody in $\mathcal{H}$, we also say that $\ell$ is a longitude of $\mathcal{H}$.

In the following, we assume $\mathcal{H}=\{H_1,H_2\}$ or $\{H\}$, and $M=\mathcal{H}_F$ is an amalgamation of $\mathcal{H}$ along $F_1$ and $F_2$ (where $F=F_1=F_2$ is $M$).

Theorem~\ref{tt3.9} characterizes a manifold $M$ to be a handlebody when $F$ is an incompressible annulus~$A$.

\begin{thm}\label{tt3.9}
(1) For $\mathcal{H}=\{H_1,H_2\}$, denote by $\mathcal{H}_A$ an amalgamation along an annulus $A$. Then $\mathcal{H}_A$ is a handlebody if and only if the core curve $\ell_i$ of $A_i$ is a longitude of $H_i$, for $i=1,2$.

(2) For $\mathcal{H}=\{H\}$ and $M=\mathcal{H}_A$, let $\ell_i$ be a core curve of $A_i$, $i=1,2$. Then $M$ is a handlebody if and only if some $\ell_i$ is a longitude for $H$, and there exists a LM-pair $(\ell_i,m_i)$ such that $\ell_j$ is disjoint from the associated disk $D_{(\ell_i,m_i)}$, for $(i,j)=(1,2)$ or $(2,1)$.
\end{thm}

\begin{proof}
We only prove (2). The proof of (1) is similar.

\textbf{Necessity.} Assume that $M$ is a handlebody. If $A$ is compressible in $M$, then one of $A_1$ and $A_2$, say $A_2$, is compressible in $H$. Let $D$ be a compression disk of $A_2$ in $H$, and $N=D\times I$ a regular neighbourhood of $D$ in $H$ with $(\partial D)\times I=A_2$. Set $H'=\overline{H-N}$, $D_0=D\times \{ 0 \}$, and $D_1=D \times \{ 1\}$. Then $M=(H'\cup_A N)_{D_0\cup D_1}$. It is clear that $M$ is a handlebody implies that $H'\cup_A N$ is a handlebody. By Proposition~\ref{pp2.6}, $\ell_1$ is a longitude of $H'$ (therefore, $H$) with respect to a meridian $m_1$ in $H'$ (in $H$), and $\ell_2=\partial D$ is disjoint from the associated disk $D_{(\ell_1,m_1)}$, see Fig.~\ref{fig9}.
\begin{figure}[!htb]
	\centering
	\includegraphics[width=0.8\textwidth]{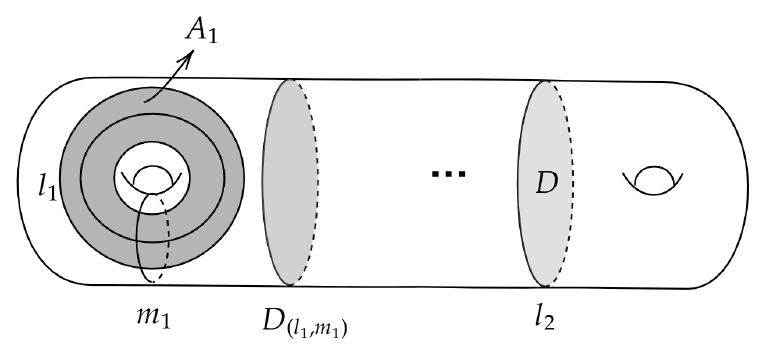}
    \caption{$\ell_2= \partial D$ is disjoint from the associated disk $D_{(\ell_1,m_1)}$.}
    \label{fig9}
\end{figure}

In the following, we assume that $A$ is incompressible in $M$. If $A$ is disjoint from an essential disk $E$ of $M$, then $M\setminus E$ has a handlebody component $H''$ with $A\subset H''$ and $g(H'')<g(M)$. Without loss of generality, we can assume that $A$ intersects all essential disks of $M$ non-empty. Choose an essential disk $D$ of $M$, such that $|A\cap D|$ is minimal over all essential disks in $M$.

\textbf{Claim~1.} \textit{$A\cap D$ has no circle component.}

Otherwise, let $\gamma$ be a circle component of $A\cap D$ which is innermost in $D$, i.e., $\gamma$ bounds a sub-disk $D'$ of $D$ with $A\cap\text{int}(D')=\emptyset$. Since $A$ is incompressible in $M$, $\gamma$ bounds a disk $D''$ in $A$. Since $M$ is irreducible, so $D'\cup_\gamma D''$ bounds a 3-ball which can be used to isotope $D$ in $M$ to reduce $|A\cap D|$ what contradict to the minimality of $|A\cap D|$.

Thus, each component $A\cap D$ is a simple arc properly embedded in $A$.

\textbf{Claim~2.} \textit{Each component $A\cap D$ is an essential arc in $A$.}

Otherwise, let $\delta$ be an arc component of $A\cap D$ which is outermost in $A$, i.e., $\delta$ cuts out of a sub-disk $\Delta$ from $A$ with $D\cap\text{int}(\Delta)=\emptyset$. $\delta$ cuts $D$ into two sub-disks $D_1$ and $D_2$. Denote $D'=D_1\cup \Delta$, and $D''=D_2\cup \Delta$. The at least one of $D'$ and $D''$, say $D'$, is an essential disk in $M$, and after an isotopy of $D$, we have $|A\cap D'|<|A\cap D|$, again contradicting to the minimality of $|A\cap D|$.

Now choose a component $\delta'$ of $A\cap D$ which is outermost in $D$, $\delta'$ cuts out of a sub-disk $\Delta'$ from $D$ with $A\cap\text{int}(\Delta')=\emptyset$. Then $\Delta'$ is a disk properly embedded in $H$. Say, $|\ell_1\cap \Delta'|=1$, and set $m_1=\partial \Delta'$. Then $\ell_1$ is a longitude for $H$ with respect to the meridian $m_1$, and $\ell_2\cap (\ell_1\cup m_1)=\emptyset$, so $\ell_2$ is disjoint from the associated disk $D_{(l_1,m_1)}$, see Fig.~\ref{fig10}.
\begin{figure}[!htb]
	\centering
	\includegraphics[width=0.8\textwidth]{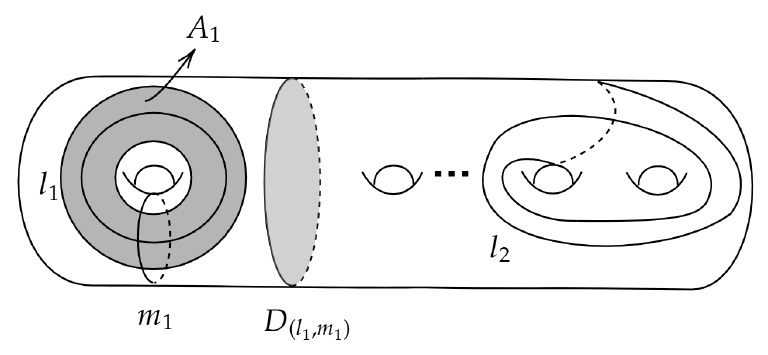}
	\caption{$l_2$ is disjoint from the associated disk $D_{(l_1,m_1)}$.} \label{fig10}
\end{figure}

\textbf{Sufficiency}. Assume that the core curve $\ell_1$ of $A_1$ is a longitude for $H_1$ with respect to a meridian curve $m_1$, and the core curve $\ell_2$ of $A_2$ is disjoint from the associated disk $D=D_{(\ell_1,m_1)}$. $D$ cuts $H$ into a solid torus $T$ and a handlebody $H'$ of genus $g(H_1)-1$, such that $(\ell_1,m_1)$ is a LM-pair for $T$. By Proposition~\ref{pp2.7}, $H'(f)=H'\cup_{A_1=A_2}T$ is a handlebody of genus $g(H')$. Note that $M$ is obtained from $H'$ by gluing the two cutting sections of $D$ together, so $M$ is a handlebody of genus $g(H')+1$ ($=g(H)$).
\end{proof}

\section{Proof of Theorem~\ref{thm1.1}} \label{sec6}




In the view of Theorem \ref{thm1.1}, we write $\mathcal{J}=\{J_1,\ldots,J_p\},\ \mathcal{D}=\{D_1,\ldots,D_p\}$, and call the pair $(\mathcal{J},\mathcal{D})$ a \textit{good JD-pair} for $(\mathcal{H},F)$. The proof of Theorem~\ref{thm1.1} is as following.

\begin{proof}
Let $F = S_{g,b}$ be the compact connected orientable surface of genus $g$ with $b$ boundary components. 
Recall that complexity of $S_{g,b}$ is defined by formula~(\ref{eqn3.1}).
It is clear that for $b\geq 1$, $C(S_{g,b})=0$ if and only if $(g,b) = (0,1)$ i.e., $F$ is a disk.  We will prove the conclusion by induction on $C(F)$. When $F$ is an annulus $A$, the conclusion follows from Theorem~\ref{tt3.9}(1) if $A$ is separating in $M$ and from Theorem~\ref{tt3.9}(2) if $A$ is non-separating in $M$. To prove by induction we assume that the conclusion is true for all compact connected orientable incompressible surfaces $F$ with $C(F)\leq k$, where $k\geq 1$.

\textbf{Necessity.} Let the handlebody $M$ be an amalgamation of $\mathcal{H}$ along $F_1$ and $F_2$ (where $F=F_1=F_2$ is incompressible in $M$) with $C(F)=k+1$.

By Theorem~\ref{tt3.3}, $F$ is $\partial$-compressible in $M$, so there exists a $\partial$-compression disk $\Delta$ in $M$, $\Delta\cap F=\gamma$ is an essential arc in $F$. Let $F'$ be the surface, and $\mathcal{H}'$ the collection of handlebodies, obtained by doing a $\partial$-compression on $F$ in $M$ along $\Delta$. Remark that $F'$ is still incompressible in $M$.

If $\gamma$ is non-separating in $F$, then $F'$ is connected. Set $\partial\Delta=\beta$. Let $\alpha$ be a simple closed curve on $F$ such that $|\alpha\cap\Delta|=1$. Then $(\alpha,\beta)$ is a LM-pair. Clearly, $C(F')=k$. By induction, there exists a JD-pair $(\mathcal{J}',\mathcal{D}')$ for $(\mathcal{H}',F')$. Set $\mathcal{J}=(\mathcal{J}',\alpha)$ and $\mathcal{D}=(\mathcal{D}',\Delta)$. It is obvious that $(\mathcal{J},\mathcal{D})$ is a JD-pair for $(\mathcal{H},F)$, the conclusion holds.

If $\gamma$ is separating in $F$, then $F'$ has two components $F_1'$ and $F_2'$. Since neither $F_1'$ nor $F_2'$ is a disk, so $C(F_1') \leq k$ and $C(F_2')\leq k$. By Theorem~\ref{tt3.3}, we may assume that $F'$ cuts $\mathcal{H}$ to a collection $\mathcal{H}'$ of handlebodies. We discuss it in the following four cases:
\begin{itemize}
\item[(1)] $F'_1$ cuts $M$ into two handlebodies $X$ and $Y'$, and $F'_2$ cuts $Y'$ into two handlebodies $Y$ and $Z$. Again by Theorem~\ref{tt3.3}, both $X\cup_{F'_1}Y$ and $Y\cup_{F'_2}Z$ are handlebodies. Set $\mathcal{V}_1=\{X,Y\}$, $\mathcal{V}_2=\{Y,Z\}$. By induction, there exists a JD-pair $(\mathcal{J}_1,\mathcal{D}_1)$ for $(X\cup Y,F_1')$ and a JD-pair $(\mathcal{J}_2,\mathcal{D}_2)$ for $(Y\cup Z,F_2')$. Set $\mathcal{J}=(\mathcal{J}_1,\mathcal{J}_2)$ and $\mathcal{D}=(\mathcal{D}_1,\mathcal{D}_2)$. It is clear that $(\mathcal{J},\mathcal{D})$ is a JD-pair for $(\mathcal{H},F)$, and $p=\operatorname{rank}(\pi_1(F))$, the conclusion holds.
\item[(2)] One of $F_1'$ and $F_2'$ is separating in $M$, and the other is non-separating in $M$, say, $F_1'$ is cutting $M$ into $X$ and $Y$, and $F_2'$ is non-separating in $Y$ (equivalently, in~$M$).
\item[(3)] Each of $F_1'$ and $F_2'$ is non-separating in $M$, but $\{F_1',F_2'\}$ is separating in $M$.
\item[(4)] Each of $F_1'$, $F_2'$ and $\{F_1',F_2'\}$ is non-separating in $M$.
\end{itemize}
The proofs of (2), (3) and (4) are all similar to that of (1).

\textbf{Sufficiency}. By the assumption, there exists a good JD-pair $(\mathcal{J},\mathcal{D})$ for $(\mathcal{H},F)$, where $\mathcal{J}=(J_1,\ldots,J_p)$, $\mathcal{D}=(D_1,\ldots,D_p)$, and $p=\operatorname{rank}(\pi_1(F))$. Note that $(\mathcal{J},\mathcal{D})$ is a good JD-pair for $(\mathcal{H},F)$ implies that $D_p\cap F$ consists of only a single arc, say $\gamma$. Thus $D_p$ is a $\partial$-compression disk for $F$ in $M$. Let $F'$ ($\mathcal{H}'$) be the surface (the collection of handlebodies) obtained by doing a $\partial$-compression on $F$ in $M$ along $D_p$. Then $C(F')=k$ and $M=\mathcal{H}_F\cong \mathcal{H}'_{F'}$. If $\gamma$ is non-separating in $F$, set $\mathcal{J}'=(J_1,\ldots,J_{p-1}),\ \mathcal{D}'=(D_1,\ldots,D_{p-1})$, it is clear that $(\mathcal{J}',\mathcal{D}')$ is a good JD-pair for $(\mathcal{H}',F')$. By the induction assumption, $M$ is a handlebody. If $\gamma$ is separating in $F$, then $F'$ has two components $F_1'$ and $F_2'$. Again we can split it into four cases as above and use similar arguments to show that $M$ is a handlebody.
This completes the proof.
\end{proof}

\section{Existence of maximal system of compression discs} \label{sec7}

\begin{thm}\label{tt3.11}
Let $M$ be an irreducible 3-manifold with reducible boundary, and $F$ a non-trivial compressible surface in $M$ with non-empty boundary. Suppose that $E$ is a compression disk of $\partial M$ in $M$ such that $E\cap F$ has no circle components. Then there exists a maximal system $\mathcal{D}$ of compression disks for $F$, such that $E\cap \mathcal{D}=\emptyset$.
\end{thm}

\begin{proof}
We only prove that there exists a compression disk $D$ for $F$, such that $E\cap D=\emptyset$. The proof of the existence of such a maximal system of compression disks for $F$ is similar.

Isotope $E$ in $M$ such that $E$ and $F$ are in general position, and $E\cap F$ has minimal components over all allowed possibilities. Let $D$ be a compression disk of $F$, such that $\partial D\subset \text{Int}(F)$, $D$ is in general position with respect to $E$, and $D\cap E$ has minimal components over all allowed possibilities. If $D\cap E=\emptyset$, the conclusion holds. In the following, we assume that $D\cap E\ne \emptyset$.

If $E \cap F = \emptyset$, we assert that $D\cap E=\emptyset$. Otherwise, each component of $D\cap E$ is a circle in both $D$ and $E$. Let $\alpha$ be a circle component of $D\cap E$ which is innermost in $E$, i.e. $\alpha$ bounds a disk $\Delta$ in $E$ with $\text{Int}(\Delta)\cap D=\emptyset$. $\alpha$ bounds a disk $\Delta'$ in $D$. Since $M$ is irreducible, the 2-sphere $\Delta\cup\Delta'$ bounds a 3-ball in $M$, thus we can push $\Delta'$ in $M$ by isotopy to cross $\Delta$, then $D$ is changed to $D'$, and $|D'\cap E|\leq |D\cap E|-1$, contradicting to the minimality of $|D\cap E|$.

In the following, we assume that $E \cap F \neq \emptyset$. Note that $E\cap F$ is a collection of pairwise disjoint  arcs (by the assumption), and each arc is properly embedded in both $E$ and $F$. Assume that $D\cap E\ne\emptyset$. We have the following claims.

{\bf Claim~1}. \textit{$D\cap E$ has no circle components.}

The proof is similar to the proof as above.

Thus, each component of $D\cap E$ is a simple arc $\alpha$ with $\partial\alpha\subset E\cap F$. Let $\beta$ be a component of $D\cap E$ which is outermost in $D$, i.e. $\beta$ cuts out of a subdisk $\Delta$ from $D$ with $\text{Int}(\Delta)\cap E=\emptyset$.

{\bf Claim~2}. \textit{$\partial\beta$ is lying in the different components of $E\cap F$.}

Otherwise, assume that $\partial\beta$ is lying in a component $\gamma$ of $E\cap F$,  $\partial\beta$ bounds an arc $\gamma'$ in $\gamma$ and $\beta\cup\gamma'$ bounds a disk $E'$ in $E$ with $\text{Int}(E')\cap D=\emptyset$. $\gamma$ cuts $D$ into two sub-disks $D'$ and $D''$. Denote $D_1=D'\cup E'$ and $D_2=D''\cup E'$. Then $\partial D_1, \partial D_2\subset F$, and at least one of $\partial D_1$ and $\partial D_2$, say the former, is essential in $F$ (otherwise, $\partial D$ would be trivial in $F$, contradicting to the choice of $D$),  therefore $D_1$ is a compression disk of $F$. It is clear that after a small isotopy of $D_1$ to $D_1'$ around $E'$, we have $|D_1'\cap E|<|D\cap E|$, again a contradiction to the minimality of $|D\cap E|$.

Now $\beta$ is a component of $D\cap E$ with $\Delta\cap E=\beta$, the two points of $\partial\beta$ are lying in different components $\alpha_1$ and $\alpha_2$, respectively. Push $F$ in $M$ to $\hat{F}$ by isotopy along $\Delta$, and $D$ is changed to $\hat{D}$, see Fig.~\ref{fig11}.
\begin{figure}[!htb]
	\centering
	\includegraphics[width=0.85\textwidth]{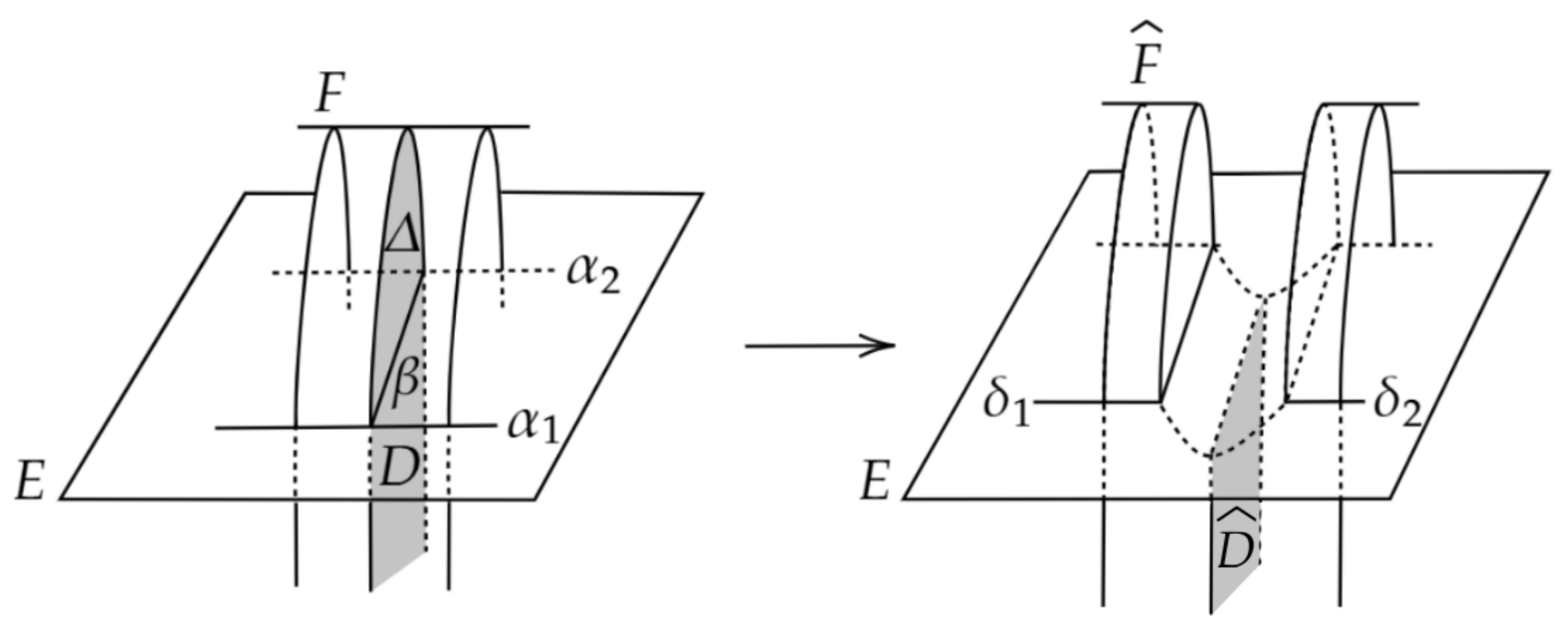}
	\caption{Push $F$ in $M$ to $\hat{F}$ by isotopy along $\Delta$.} \label{fig11}
\end{figure}
Note that the two components $\alpha_1$ and $\alpha_2$ are changed into $\delta_1$ and $\delta_2$, and all other components of $F\cap E$ keep unchanged. Thus $|E\cap \hat{F}|=|E\cap F|$, but $|\hat{D}\cap E|=|D\cap E|-1$, again contradicting to the choice of $D$.
\end{proof}


\begin{cor} \label{cc3.13}
Let $M$ be an irreducible 3-manifold with reducible boundary, and $F$ a non-trivial incompressible surface in $M$ with non-empty boundary. Suppose that $E$ is a compression disk of $\partial M$ in $M$ such that $E\cap F=\emptyset$, and $F(\alpha)$ is a band sum of $F$ along $\alpha$ such that $\alpha$ intersects $E$ in a single point. Then $F(\alpha)$ is incompressible in $M$.
\end{cor}

\begin{proof}
By assumption, $F(\alpha)$ is a band sum of $F$ along $\alpha$ such that $\alpha$ intersects $E$ in a single point, thus $F(\alpha)$ and $E$ intersect with only one component of arc. If $F(\alpha)$ is compressible in $M$, choose a compressing disk $D$ of $F(\alpha)$ in $M$, such that $D$ is in general position with $E$, and $|D\cap E|$ is minimal over all such compressing disks of $F(\alpha)$ in $M$. Then $D\cap E\ne\emptyset$.  Otherwise, $D\cap E=\emptyset$, since $\partial D$ does not bound a disk in $F(\alpha)$, so $\partial D$ does not bound a disk in $F$, which implies that $F$ is a compressible in $M$, contradicting to the assumption.

Similar to the situation in the proof of Theorem \ref{tt3.9}, $D\cap E$ has no circle components. Since $F(\alpha)\cap E$ consists of only a single component $\beta$, the two ends of each component of $D\cap E$ is an simple arc $\gamma$ with $\partial\gamma\subset\beta$. By Claim~2 in the proof of Theorem~\ref{tt3.9}, this would derived a contraction that $|D\cap E|$ is minimal. Therefore such a $D$ cannot exist. This completes the proof.
\end{proof}

\end{document}